\documentclass[11pt]{amsart}
\usepackage{amsopn}
\usepackage{amssymb, amscd}
\usepackage{multirow}

\newcommand{\nc}{\newcommand}

\nc{\fg}{\mathfrak{f} } \nc{\vg}{\mathfrak{v} } \nc{\wg}{\mathfrak{w} }
\nc{\zg}{\mathfrak{z} } \nc{\ngo}{\mathfrak{n} } \nc{\kg}{\mathfrak{k} }
\nc{\mg}{\mathfrak{m} } \nc{\bg}{\mathfrak{b} } \nc{\ggo}{\mathfrak{g} }
\nc{\ggob}{\overline{\mathfrak{g}} } \nc{\sog}{\mathfrak{so} }
\nc{\sug}{\mathfrak{su} } \nc{\spg}{\mathfrak{sp} } \nc{\slg}{\mathfrak{sl} }
\nc{\glg}{\mathfrak{gl} } \nc{\cg}{\mathfrak{c} } \nc{\rg}{\mathfrak{r} }
\nc{\hg}{\mathfrak{h} } \nc{\tg}{\mathfrak{t} } \nc{\ug}{\mathfrak{u} }
\nc{\dg}{\mathfrak{d} } \nc{\ag}{\mathfrak{a} } \nc{\pg}{\mathfrak{p} }
\nc{\sg}{\mathfrak{s} } \nc{\affg}{\mathfrak{aff} } \nc{\qg}{\mathfrak{q} }

\nc{\pca}{\mathcal{P}} \nc{\nca}{\mathcal{N}} \nc{\lca}{\mathcal{L}}
\nc{\oca}{\mathcal{O}} \nc{\mca}{\mathcal{M}} \nc{\tca}{\mathcal{T}}
\nc{\aca}{\mathcal{A}} \nc{\cca}{\mathcal{C}} \nc{\gca}{\mathcal{G}}
\nc{\sca}{\mathcal{S}} \nc{\hca}{\mathcal{H}} \nc{\bca}{\mathcal{B}}
\nc{\dca}{\mathcal{D}} \nc{\val}{\operatorname{val}}

\nc{\vp}{\varphi} \nc{\ddt}{\frac{d}{dt}} \nc{\dds}{\frac{d}{ds}}
\nc{\dpar}{\frac{\partial}{\partial t}} \nc{\im}{\mathrm{i}}

\nc{\SO}{\mathrm{SO}} \nc{\Spe}{\mathrm{Sp}} \nc{\Sl}{\mathrm{SL}}
\nc{\SU}{\mathrm{SU}} \nc{\Or}{\mathrm{O}} \nc{\U}{\mathrm{U}} \nc{\Gl}{\mathrm{GL}}
\nc{\Se}{\mathrm{S}} \nc{\Cl}{\mathrm{Cl}} \nc{\Spein}{\mathrm{Spin}}
\nc{\Pin}{\mathrm{Pin}} \nc{\G}{\mathrm{GL}_n(\RR)} \nc{\g}{\mathfrak{gl}_n(\RR)}

\nc{\RR}{{\Bbb R}} \nc{\HH}{{\Bbb H}} \nc{\CC}{{\Bbb C}} \nc{\ZZ}{{\Bbb Z}}
\nc{\FF}{{\Bbb F}} \nc{\NN}{{\Bbb N}} \nc{\QQ}{{\Bbb Q}} \nc{\PP}{{\Bbb P}} \nc{\OO}{{\Bbb O}}

\nc{\vs}{\vspace{.2cm}} \nc{\vsp}{\vspace{1cm}} \nc{\ip}{\langle\cdot,\cdot\rangle}
\nc{\ipp}{(\cdot,\cdot)} \nc{\la}{\langle} \nc{\ra}{\rangle} \nc{\unm}{\tfrac{1}{2}}
\nc{\unc}{\tfrac{1}{4}} \nc{\und}{\tfrac{1}{16}} \nc{\no}{\vs\noindent}
\nc{\lam}{\Lambda^2(\RR^n)^*\otimes\RR^n} \nc{\tangz}{{\rm T}^{\rm Zar}}
\nc{\nor}{{\sf n}}  \nc{\mum}{/\!\!/} \nc{\kir}{/\!\!/\!\!/}
\nc{\Ri}{\tfrac{4\Ric_{\mu}}{||\mu||^2}} \nc{\ds}{\displaystyle}
\nc{\ben}{\begin{enumerate}} \nc{\een}{\end{enumerate}} \nc{\f}{\frac}
\nc{\lb}{[\cdot,\cdot]} \nc{\isn}{\tfrac{1}{||v||^2}}
\nc{\gkp}{(\ggo=\kg\oplus\pg,\ip)} \nc{\ukh}{(\ug=\kg\oplus\hg,\ip)}
\nc{\tgkp}{(\tilde{\ggo}=\kg\oplus\pg,\ip)}
\nc{\wt}{\widetilde} \nc{\mm}{M}
\nc{\iop}{\mathtt{i}} \nc{\jop}{\mathtt{j}}

\nc{\Hess}{\operatorname{Hess}} \nc{\ad}{\operatorname{ad}}
\nc{\Ad}{\operatorname{Ad}} \nc{\rank}{\operatorname{rank}}
\nc{\Irr}{\operatorname{Irr}} \nc{\End}{\operatorname{End}}
\nc{\Aut}{\operatorname{Aut}} \nc{\Inn}{\operatorname{Inn}}
\nc{\Der}{\operatorname{Der}} \nc{\Ker}{\operatorname{Ker}}
\nc{\Iso}{\operatorname{Iso}} \nc{\Diff}{\operatorname{Diff}}
\nc{\Lie}{\operatorname{L}} \nc{\tr}{\operatorname{tr}} \nc{\dif}{\operatorname{d}}
\nc{\sen}{\operatorname{sen}} \nc{\modu}{\operatorname{mod}}
\nc{\CRic}{\operatorname{PP}} \nc{\Cric}{\operatorname{P}} \nc{\Ricci}{\operatorname{Ric}}
\nc{\sym}{\operatorname{sym}} \nc{\herm}{\operatorname{herm}} \nc{\symac}{\operatorname{sym^{ac}}}
\nc{\symc}{\operatorname{sym^{c}}} \nc{\scalar}{\operatorname{Scal}}
\nc{\grad}{\operatorname{grad}} \nc{\ricci}{\operatorname{Rc}}
\nc{\Nor}{\operatorname{Norm}}  \nc{\ricc}{\operatorname{Rc^{c}}}
\nc{\Ricc}{\operatorname{Ric^{c}}} \nc{\ricac}{\operatorname{Rc^{ac}}}
\nc{\Ricac}{\operatorname{Ric^{ac}}} \nc{\Riem}{\operatorname{Rm}}
\nc{\riccig}{\operatorname{ric^{\gamma}}} \nc{\Rin}{\operatorname{M}}
\nc{\Le}{\operatorname{L}} \nc{\tang}{\operatorname{T}}
\nc{\level}{\operatorname{level}} \nc{\rad}{\operatorname{r}}
\nc{\abel}{\operatorname{ab}} \nc{\CH}{\operatorname{CH}}
\nc{\mcc}{\operatorname{mcc}} \nc{\Adj}{\operatorname{Adj}}
\nc{\Order}{\operatorname{O}}  \nc{\inj}{\operatorname{inj}} \nc{\proy}{\operatorname{pr}}
\nc{\vol}{\operatorname{vol}} \nc{\Diag}{\operatorname{Dg}}
\nc{\Spec}{\operatorname{Spec}} \nc{\Ima}{\operatorname{Im}} \nc{\Rea}{\operatorname{Re}}
\nc{\spann}{\operatorname{span}}

\theoremstyle{plain}
\newtheorem{theorem}{Theorem}[section]
\newtheorem{proposition}[theorem]{Proposition}
\newtheorem{corollary}[theorem]{Corollary}
\newtheorem{lemma}[theorem]{Lemma}

\theoremstyle{definition}

\theoremstyle{remark}

\newtheorem{example}[theorem]{Example}

\title{Distinguished $G_2$-structures on solvmanifolds}

\author{Jorge Lauret}

\address{Universidad Nacional de C\'ordoba, FaMAF and CIEM, 5000 C\'ordoba, Argentina}
\email{lauret@famaf.unc.edu.ar}

\thanks{This research was partially supported by grants from CONICET, FONCYT and SeCyT (Universidad Nacional de C\'ordoba)}

\begin{document}

\maketitle

\begin{abstract}
Among closed $G_2$-structures there are two very distinguished classes: {\it Laplacian solitons} and {\it Extremally Ricci-pinched} $G_2$-structures.  We study the existence problem and explore possible interplays between these concepts in the context of left-invariant $G_2$-structures on solvable Lie groups.  Also, some Ricci pinching properties of $G_2$-structures on solvmanifolds are obtained, in terms of the extremal values and points of the functional $F=\frac{\scalar^2}{|\Ricci|^2}$, $0< F < 7$.  Many natural open problems have been included.
\end{abstract}

\tableofcontents

\section{Introduction}\label{intro}

Our main motivation in this article is the following heuristic though very natural and intriguing question, which we borrowed from the first page of Besse's book \cite{Bss} and adapted to $G_2$-geometry:

\begin{quote}
Given a  $7$-dimensional differentiable manifold $M$, are there any best (or nicest, or most distinguished) $G_2$-structures on $M$?
\end{quote}

The question remains natural when restricted to special kinds of manifolds or particular classes of $G_2$-structures, like the set of all $G_2$-structures with the same associated metric, left-invariant $G_2$-structures on a given Lie group, etc.  The meaning of the adjectives in the question are of course part of the problem, and any good candidate is expected to be weak enough to allow existence results but also sufficiently strong to imply some kind of uniqueness or finiteness results.

As a first reduction, we consider closed $G_2$-structures, but there are many other reasonable and natural special classes to start with.  We have included in an appendix (see Section \ref{dist-G2}) the definition of several of them, as well as a diagram describing the inclusion relationships between such classes (see Figure \ref{g2-fig}).  No topological obstruction on $M$ to admit a closed $G_2$-structure is known, other than the ones for admitting just a $G_2$-structure, i.e.\ orientable and spin.

In the case when a $G_2$-structure $\vp$ is closed, the only torsion that survives is contained in a $2$-form $\tau$, and the starting situation can be described as follows:
$$
d\vp=0, \qquad \tau=-\ast d\ast\vp, \qquad d\ast\vp=\tau\wedge\vp, \qquad d\tau=\Delta\vp,
$$
where $\ast$ and $\Delta$ denote the Hodge star and Laplacian operator, respectively, defined by the metric attached to $\vp$.

Among closed $G_2$-structures, one finds two concepts which are both distinguished but from points of view of a very different taste:

\begin{itemize}
\item {\it Laplacian solitons}: $d\tau=c\vp+\lca_X\vp$ for some $c\in\RR$ and $X\in\mathcal{X}(M)$.
\item[ ]
\item {\it Extremally Ricci-pinched}: $d\tau=\frac{1}{6}|\tau|^2\vp+\frac{1}{6}\ast(\tau\wedge\tau)$.
\end{itemize}

In this paper, we mainly work in the homogeneous setting (see \cite{LF} for further information); more specifically, in the context of left-invariant $G_2$-structures on solvable Lie groups (or {\it solvmanifolds}).  We aim to overview what is known on the existence of the above two special classes of $G_2$-structures and explore possible interplays.  Diverse open problems have been included throughout the paper.

Any $G_2$-structure or metric on a Lie group is always assumed to be left-invariant.

\subsection{Solitons}\label{sol}
The space $\gca$ of all $G_2$-structures on a given $7$-dimensional manifold $M$ is an open cone in $\Omega^3M$, whose equivalence classes are $\Diff(M)$-orbits.  Assume that at each $\vp\in\gca$, we have a preferred direction $q(\vp)\in\Omega^3M$, an optimal `direction of improvement' in some sense (e.g.\ the gradient of a natural functional on $\gca$).  It is therefore reasonable to consider an element $\vp\in\gca$ distinguished when $q(\vp)$ is tangent to its equivalence class (up to scaling), i.e.\
\begin{equation}\label{sol1}
  q(\vp)\in T_\vp\left(\RR^*\Diff(M)\cdot\vp\right).
\end{equation}
Heuristically, it is like such a $\vp$ is nice enough that it does not need to be improved.  A $G_2$-structure for which condition \eqref{sol1} holds will be called a $q$-{\it soliton}.  It is easy to see that if $q$ is $\Diff(M)$-equivariant, then the following conditions are equivalent:

\begin{itemize}
  \item $\vp$ is a $q$-soliton.
  \item[ ]
  \item $q(\vp)=c\vp+\lca_X\vp$ for some $c\in\RR$, $X\in\mathfrak{X}(M)$.
  \item[ ]
  \item The solution $\vp(t)$ starting at $\vp$ to the corresponding geometric flow 
  $$
  \dpar\vp(t)=q(\vp(t)),
  $$ 
is {self-similar}, i.e.\ $\vp(t)=c(t)f(t)^*\vp$ for some $c(t)\in\RR$ and $f(t)\in\Diff(M)$.
\end{itemize}
The $q$-soliton is said to be {\it expanding}, {\it steady} or {\it shrinking} if $c>0$, $c=0$ or $c<0$, respectively.  The corresponding self-similar solutions are respectively immortal, eternal and ancient if $q(a\vp)=a^\alpha q(\vp)$ for any $a\in\RR^*$ and some fixed $\alpha<1$ (see \cite[Section 4.4]{BF}).

We consider in this paper the direction $q(\vp)=\Delta\vp$, which determines the so called {\it Laplacian solitons} and the Laplacian flow introduced by Bryant in \cite{Bry}.  Many other types of $q$-solitons have also been studied in the literature, see for example \cite{WssWtt1,WssWtt2,Grg,BggFin}.

We next list all the results on Laplacian solitons in the literature that we are aware of:

\begin{itemize}
\item \cite[Corollary 1]{Lin} There are no compact shrinking Laplacian solitons, and the only compact steady Laplacian solitons are the torsion-free $G_2$-structures (see also \cite[Proposition 9.4]{Lty} for a shorter proof in the closed case).
\item[ ]
\item Any nearly parallel $G_2$-structure $\vp$ satisfies $\Delta\vp=c^2\vp$ and so it is a coclosed expanding Laplacian soliton.  Examples are given by the round and squashed spheres (see
    \cite[Section 4.1]{WssWtt2}).
\item[ ]
\item \cite[Section 6]{KrgMcKTsu} Examples of non-compact expanding coclosed Laplacian solitons which are not nearly parallel.  However, they still are all {\it eigenforms} (i.e.\ $\Delta\vp=c\vp$ for some $c\in\RR$).
\item[ ]
\item \cite[Proposition 9.1]{Lty} The only compact and closed Laplacian solitons which are eigenforms are the torsion-free $G_2$-structures.
\item[ ]
\item \cite[Section 7]{BF} A closed $G_2$-structure on a nilpotent Lie group which is an expanding Laplacian soliton and is not an eigenform was found.
\item[ ]
\item \cite{Ncl} Closed expanding Laplacian solitons were exhibited on seven of the twelve nilpotent Lie groups admitting a closed $G_2$-structure.  There is even a one-parameter family of pairwise non-homothetic closed Laplacian solitons on one of them.
\item[ ]
\item \cite{LF} Homogeneous Laplacian solitons are studied using the algebraic soliton approach.  Many continuous families of expanding Laplacian solitons on almost-abelian Lie groups were given (see \cite[Section 5.2]{LF}).
\item[ ]
\item \cite[Section 4]{LS-ERP} Examples of steady and shrinking closed Laplacian solitons were found on solvmanifolds by using coupled $\SU(3)$-structures.
\item[ ]
\item \cite{FinRff3} Closed expanding Laplacian solitons were found on solvmanifolds from symplectic half-flat $\SU(3)$-structures.
\end{itemize}

The following are open questions on Laplacian solitons:

\begin{itemize}
\item Are there compact and closed expanding Laplacian solitons?
\item[ ]
\item Are there compact expanding Laplacian solitons other than nearly parallel $G_2$-structures?
\end{itemize}

\subsection{Extremally Ricci pinched $G_2$-structures}\label{erp}
The following nice interplay between the metric and the torsion $2$-form of a closed $G_2$-structure was discovered by R. Bryant.  Let $\scalar$ and $\Ricci$ denote the scalar and Ricci curvature of the metric attached to a $G_2$-structure.

\begin{theorem}\cite[Corollary 3]{Bry}\label{bryant}
If $\vp$ is a closed $G_2$-structure on a compact manifold $M$, then
$$
\int_M \scalar^2 \ast 1 \leq 3\int_M |\Ricci|^2 \ast 1,
$$
and equality holds if and only if $\; d\tau = \frac{1}{6}|\tau|^2\vp + \frac{1}{6}\ast(\tau\wedge\tau)$.
\end{theorem}

The special $G_2$-structures for which equality holds were called {\it extremally Ricci-pinched} (ERP for short) in \cite[Remark 13]{Bry}.   Notice the factor of $3$ on the right hand side, much smaller than the factor of $7$ provided by the Cauchy-Schwartz inequality, only attained at Einstein metrics.

As far as we know, there are only two examples of ERP $G_2$-structures in the literature and they are both homogeneous: the first one was given in \cite[Example 1]{Bry} on the homogeneous space $\Sl_2(\CC)\ltimes\CC^2/\SU(2)$ (and on any compact quotient by a lattice), which also has a presentation as a $G_2$-structure on the solvable Lie group given in \cite[Examples 4.13, 4.10]{LS-ERP}, and the one found on a unimodular solvable Lie group in \cite[Example 4.7]{LS-ERP}.  Surprisingly (or not), both examples are also steady Laplacian solitons.  We do not know if there could be an interplay between the two notions.

Motivated by Theorem \ref{bryant}, we consider the invariant (up to isometry and scaling) functional
$$
F:=\frac{\scalar^2}{|\Ricci|^2}, \qquad 0\leq F\leq 7,
$$
on the space of all non-flat homogeneous closed $G_2$-structures (recall that a Ricci flat homogeneous Riemannian manifold is necessarily flat; see \cite{AlkKml}).  Note that after integrating on $M$ in the compact case, $F\leq 3$ by Theorem \ref{bryant}.  On the other hand, we know that $F<7$ always since no solvable Lie group admits an Einstein closed (non-parallel) $G_2$-structure (see \cite{FrnFinMnr2}) and the Alekseevskii Conjecture, asserting that any homogeneous Einstein metric of negative scalar curvature is isometric to a solvmanifold, has recently been proved in dimension $7$ (see \cite{ArrLfn}).

Within the class of closed $G_2$-structures on almost-abelian Lie groups, $F\leq 1$ and equality holds precisely at non-nilpotent expanding Laplacian solitons (see \cite[Section 5]{LF}).  At some point it was not unreasonable to expect that $F\leq 3$ would also hold in the homogeneous case.  However, we found in \cite{LS-ERP} a curve $\vp_t$, $\unc\leq t\leq 1$, of closed $G_2$-structures on pairwise non-isomorphic solvmanifolds such that $F(\vp_t)$ is strictly decreasing and
$$
F(\vp_{1/4})=\tfrac{81}{17} \; (\approx 4.76) \; > \;  3=F(\vp_1).
$$
Furthermore, $\vp_t$ is a shrinking Laplacian soliton for any $\unc\leq t<1$ and $\vp_1$ is the ERP steady Laplacian soliton given by
Bryant.  It is therefore natural to wonder about what would be the `extremally Ricci pinched' $G_2$-structures in the homogeneous case:

\begin{quote}
What is the value of $\sup F$ and its meaning?  Is it a maximum?  Are the maximal $G_2$-structures distinguished in some sense?
\end{quote}

We study the behavior of the functional $F$ on solvmanifolds in Section \ref{RP-g2-sec}, after giving a summary on what is known on Ricci pinching of solvmanifolds in the Riemannian case in Section \ref{RP-sec}.

\vs \noindent {\it Acknowledgements.} The author is very grateful to the organizers of the `Workshop on $G_2$ Manifolds and Related Topics', August 21 - 25, 2017 and to The Fields Institute for the great hospitality.  The author would also like to thank Ramiro Lafuente for very helpful comments.

\section{The space of closed $G_2$-structures on solvmanifolds}\label{preli}

We fix a $7$-dimensional real vector space $\sg$ endowed with a basis $\{ e_1,\dots,e_7\}$ and the positive $3$-form
$$
\vp:=e^{127}+e^{347}+e^{567}+e^{135}-e^{146}-e^{236}-e^{245},
$$
whose associated inner product $\ip$ is the one making the basis $\{ e_i\}$ orthonormal.  Let $\sca\subset\Lambda^2\sg^*\otimes\sg$ denote the algebraic subset of all Lie brackets on $\sg$ which are solvable.  Each $\mu\in\sca$ will be identified with the left-invariant $G_2$-structure determined by $\vp$ on the simply connected solvable Lie group $S_\mu$ with Lie algebra $(\sg,\mu)$:
$$
\mu \longleftrightarrow (S_\mu,\vp).
$$
In this way, the isomorphism class $\Gl_7(\RR)\cdot\mu$ stands for the set of all left-invariant $G_2$-structures on $S_\mu$:
$$
(S_{h\cdot\mu},\vp) \longleftrightarrow (S_\mu,\vp(h\cdot,h\cdot,h\cdot)), \qquad\forall h\in\Gl_7(\RR).
$$
Note that $h^{-1}$ is an isomorphism determining an equivalence between these two Lie groups endowed with $G_2$-structures.  Recall that any $G_2$-structure or metric on a Lie group is assumed to be left-invariant.

Thus any two Lie brackets in the same $G_2$-orbit are equivalent as $G_2$-structures, and if they are in the same $\Or(7)$-orbit then they are isometric as Riemannian metrics.  Both converse assertions hold for completely real Lie brackets.  We note that the orbit $\Or(7)\cdot\mu$ consists of all the $G_2$-structures on $S_\mu$ defining some fixed metric.

By intersecting $\sca$ with the linear subspace $\{\mu\in\Lambda^2\sg^*\otimes\sg:d_\mu\vp=0\}$, one obtains the $G_2$-invariant algebraic subset
$$
\sca_{closed} := \{\mu\in\sca:d_\mu\vp=0\}.
$$
The space $\sca_{closed}$ therefore parameterizes the set of all closed $G_2$-structures on solvmanifolds.  Note that a Lie group $S_\mu$ admits a closed $G_2$-structure if and only if the orbit $\Gl_7(\RR)\cdot\mu$ meets $\sca_{closed}$ (or equivalently, the above linear subspace).  We do not know much about the topology of the cone $\Gl_7(\RR)\cdot\mu\cap\sca_{closed}$ of all closed $G_2$-structures on a given Lie group: is it connected?  Is its intersection with a sphere connected?

Recall that for each $\mu\in\sca_{closed}$, the only torsion form that survives is the $2$-form $\tau_\mu\in\Lambda^2\sg^*$ given by
$$
\tau_\mu=-\ast d_\mu\ast\vp, \qquad d_\mu\ast\vp=\tau_\mu\wedge\vp.
$$
We also consider the $G_2$-invariant subset of torsion-free $G_2$-structures,
$$
\sca_{tf}:= \{\mu\in\sca:d_\mu\vp=0, \; d_\mu\ast\vp=0\} = \left\{\mu\in\sca_{closed}:\tau_\mu=0\right\},
$$
and the $\Or(7)$-invariant subset
$$
\sca_{flat} := \left\{\mu\in\sca:(S_\mu,\ip)\;\mbox{is flat}\right\} = \left\{\mu\in\sca:\scalar_\mu=0\right\}.
$$
Since the scalar curvature of $\mu$ (i.e.\ of $(S_\mu,\ip)$) equals $\scalar_\mu=-\unm|\tau_\mu|^2$, we obtain that
$$
\sca_{tf} = \sca_{closed}\cap\sca_{flat}.
$$

\subsection{Nilpotent case}\label{nilp-preli}
There are exactly twelve nilpotent Lie algebras admitting a closed $G_2$-structure (see \cite{CntFrn}).  Thus the space $\sca_{closed}$ meets twelve nilpotent  $\Gl(\sg)$-orbits, say $\Gl(\sg)\cdot\mu_1,\dots,\Gl(\sg)\cdot\mu_{12}$.  In \cite{FrnFinMnr}, the authors classified which of these twelve Lie groups admit a closed $G_2$-structure which is in addition a Ricci soliton (called {\it nilsolitons} in the nilpotent case), and in \cite{Ncl}, the existence of closed Laplacian solitons was studied.  The following information has been extracted from these two articles (we use the same enumeration of the algebras):

\begin{itemize}
  \item $\mu_1$: This is the abelian Lie algebra and so $S_{\mu_1}=\RR^7$ admits a unique $G_2$-structure up to equivalence which is torsion-free.

  \item $\mu_2$: The Lie group $S_{\mu_2}$ admits a unique closed $G_2$-structure up to equivalence and scaling, which is a nilsoliton and also a Laplacian soliton.

  \item $\mu_3$: There exists a curve of closed Laplacian solitons on $S_{\mu_3}$ which are not nilsolitons; however, there are no nilsoliton closed $G_2$-structures on this group.

  \item $\mu_4$: The Lie group $S_{\mu_4}$ admits a pairwise non-equivalent one-parameter family of closed $G_2$-structures, among which there are a nilsoliton and a (different) Laplacian soliton.

  \item $\mu_5$: $S_{\mu_5}$ does not admit any closed $G_2$-structure which is a nilsoliton. There is though a Laplacian soliton belonging to a curve of closed $G_2$-structures.

  \item $\mu_6$: The Lie group $S_{\mu_6}$ admits a curve of closed $G_2$-structures, one of them being a nilsoliton and another one a Laplacian soliton.

  \item $\mu_7$: $S_{\mu_7}$ does not admit any closed $G_2$-structure which is a nilsoliton. However, there exists a curve of closed $G_2$-structures containing a Laplacian soliton.

  \item $\mu_8,\mu_9,\mu_{11}$: None of these Lie groups admit a closed $G_2$-structure which is a nilsoliton.

  \item $\mu_{10}$: The existence of a nilsoliton closed $G_2$-structure in $S_{\mu_{10}}$ is still open.

  \item $\mu_{12}$: The Lie group $S_{\mu_{12}}$ admits a closed $G_2$-structure which is also a nilsoliton.
\end{itemize}

The existence of closed Laplacian solitons on the Lie groups $S_{\mu_8},\dots,S_{\mu_{12}}$ remains open.

\subsection{Almost-abelian case}\label{alm-abel-preli}
Closed $G_2$-structures in the class of {\it almost-abelian} Lie algebras (i.e.\ with a codimension-one abelian ideal) were studied in \cite[Section 5]{LF}, we refer the reader there for further information.  One attaches to each matrix $A\in\glg_6(\RR)$ a Lie bracket $\mu_A\in\sca$ as follows: relative to a fixed orthonormal basis $\{ e_1,\dots,e_7\}$, $\ngo:=\spann\{ e_1,\dots,e_6\}$ is an abelian ideal for $\mu_A$ and $\ad_{\mu_A}{e_7}|_{\ngo}=A$.

We have that $\mu_A\in\sca_{closed}$ if and only if $A\in\slg_3(\CC)\subset\glg_6(\RR)$, where the complex structure defining $\slg_3(\CC)$ is $Je_i=e_{i+1}$, $i=1,3,5$, and $\mu_A\in\sca_{tf}$ if and only if $A\in\sug(3)$.  It is easy to see that $\mu_B\in\Gl(\sg)\cdot\mu_A$ for $A,B\in\slg_3(\CC)$ if and only if $B\in\RR^*\Sl_3(\CC)\cdot A$, where the last action is by conjugation.  This implies that every almost-abelian Lie algebra admitting a closed $G_2$-structure is isomorphic to $\mu_A$ for some matrix $A$ in the following list:
$$
\left[\begin{matrix} \alpha&&\\ &\beta& \\ &&\gamma \end{matrix}\right],  \left[\begin{matrix} \alpha&1&\\ &\alpha& \\ &&-2\alpha \end{matrix}\right],  \left[\begin{matrix} 0&1&\\ &0& \\ &&0 \end{matrix}\right], \left[\begin{matrix} 0&1&\\ &0&1 \\ &&0 \end{matrix}\right],
\left[\begin{matrix} \im &&\\ &a\im & \\ &&b\im  \end{matrix}\right], \left[\begin{matrix} \im&1&\\ &\im& \\ &&-2\im \end{matrix}\right],
$$
where $\alpha,\beta,\gamma\in\CC$, $\alpha+\beta+\gamma=0$, $|\alpha|=1$, $\alpha\ne\pm\im$ and $a,b\in\RR$, $1+a+b=0$.  The two nilpotent matrices in the middle define groups isomorphic to $S_{\mu_2}$ and $S_{\mu_6}$, respectively (see Section \ref{nilp-preli}).  Moreover, each Lie group $S_{\mu_A}$ admits an $\Sl_3(\CC)$-orbit of closed $G_2$-structures up to scaling and each $\SU(3)$-orbit consists of pairwise equivalent structures.  Thus there are continuous families of closed $G_2$-structures depending on many parameters on most of these Lie groups (see e.g.\ \cite[Example 5.9]{LF}).

It is easy to see that if $A=S+N$ for $A,S,N\in\slg_3(\CC)$, where $S$ is semisimple, $N$ nilpotent and $[S,N]=0$, then $\mu_S,\mu_N\in\overline{\RR^*\Sl_3(\CC)\cdot\mu_A}\subset\sca_{closed}$.  It is worth observing that any kind of geometric quantity associated to closed $G_2$-structures depends continuously on the Lie bracket $\mu\in\sca_{closed}$, so $\mu_S$ and $\mu_N$ inherit any property that $\mu_A$ may have.  This can also be used to study pinching curvature properties (see \cite[Section 3.3]{BF} and the next sections).

It is proved in \cite[Proposition 5.22]{LF} that $\mu_A$ is a Laplacian soliton for any normal matrix $A\in\slg_3(\CC)$ (see \cite[Propositions 5.22, 5.27]{LF} for the Laplacian soliton conditions for a nilpotent $A$).

\section{Ricci pinching of solvmanifolds}\label{RP-sec}

In this section, we give a short overview on Ricci pinching of solvmanifolds.  We refer to \cite{RP} for a more detailed treatment with a complete list of references.

We fix an $n$-dimensional real vector space $\sg$ endowed with an inner product $\ip$.  Let $\sca\subset\Lambda^2\sg^*\otimes\sg$ denote the algebraic subset of all Lie brackets on $\sg$ which are solvable.  Given $\mu\in\sca$, its isomorphism class $\Gl(\sg)\cdot\mu$ can be identified with the set of all left-invariant metrics on the corresponding simply connected solvable Lie group $S_\mu$ in the following way:
$$
(S_{h\cdot\mu},\ip) \longleftrightarrow (S_\mu,\la h\cdot,h\cdot\ra), \qquad\forall h\in\Gl(\sg).
$$
Consider the following $\Gl(\sg)$-invariant subsets of $\sca$:
\begin{align*}
\sca_{\im\RR} :=& \left\{\mu\in\sca:\Spec(\ad_\mu{X})\subset\im\RR, \;\forall X\in\sg\right\}, \\
\sca_{\RR}  :=& \left\{\mu\in\sca:\mbox{either}\, \ad_\mu{X} \, \mbox{is nilpotent or}\, \Spec(\ad_\mu{X})\nsubseteq\im\RR, \;\forall X\in\sg\right\}, \\
\sca_{c\RR} :=& \left\{\mu\in\sca:\Spec(\ad_\mu{X})\subset\RR, \;\forall X\in\sg\right\}, \\
\sca_{unim} :=& \left\{\mu\in\sca:\tr{\ad_\mu X}=0,  \;\forall X\in\sg\right\}, \\
\nca :=& \left\{\mu\in\sca: \mu \; \mbox{is nilpotent}\right\},
\end{align*}
where $\Spec(\ad_\mu{X})$ is the set of eigenvalues of the operator $\ad_\mu{X}$.  

The Lie algebras in $\sca_{\im\RR}$ and $\sca_\RR$ are called of {\it imaginary} and {\it real type}, respectively (see e.g.\ \cite[Section 3]{BhmLfn}).  The closed subset $\sca_{c\RR}$ is known in the literature as the class of {\it completely real} or {\it completely solvable} Lie algebras, and $\sca_{unim}$ is the subset of {\it unimodular} solvable Lie algebras.  It easily follows that $\sca_{\im\RR}$ is closed, $\sca_{\RR}\smallsetminus\nca$ is open in $\sca$ and $\sca_{\im\RR}\cap\sca_\RR=\nca$.  We also consider the subset
$$
\sca_{flat} := \left\{\mu\in\sca:(S_\mu,\ip)\;\mbox{is flat}\right\}.
$$
The following inclusions hold,
$$
\{ 0\}\subset\sca_{flat}\subset\Gl(\sg)\cdot\sca_{flat}\subset\sca_{\im\RR}\subset\sca_{unim}, \qquad \{ 0\}\subset\nca\subset\sca_{c\RR}\subset\sca_\RR.
$$
and the following lemma will be very useful.

\begin{lemma}\label{ramiro}\cite[Lemma 3.4]{BhmLfn}
If $\mu\in\sca_\RR$ then $\overline{\Gl(\sg)\cdot\mu}\cap\sca_{flat}=\{ 0\}$.
\end{lemma}

From our point of view, the Ricci pinching is captured by the extremal values of the functional
$$
F:\sca\smallsetminus\sca_{flat} \longrightarrow \RR, \qquad F(\mu):=\frac{\scalar_\mu^2}{|\Ricci_\mu|^2},
$$
where $\scalar_\mu$ and $\Ricci_\mu$ are respectively the scalar curvature and Ricci operator of $\mu$ (recall that $\mu\leftrightarrow(S_\mu,\ip)$).  Note that $F$ is invariant up to isometry and scaling; in particular, $F$ is $\Or(\sg)$-invariant.  Since $\scalar_\mu=0$ if and only if $\mu\in\sca_{flat}$, one obtains from the Cauchy-Schwartz inequality that
$$
0<F(\mu)\leq n, \qquad\forall\mu\in\sca\smallsetminus\sca_{flat},
$$
with $F(\mu)=n$ if and only if $\mu$ is Einstein.  For each $\mu\in\sca$ we define,
$$
m_\mu:=\inf F(\Gl(\sg)\cdot\mu), \qquad M_\mu:=\sup F(\Gl(\sg)\cdot\mu),
$$
that is, the infimum and supremum of $F$ among all left-invariant metrics on the Lie group $S_\mu$.  It follows that
$$
(m_\mu,M_\mu)\subset F(\Gl(\sg)\cdot\mu)\subset F\left(\overline{\Gl(\sg)\cdot\mu}\right)\subset[m_\mu,M_\mu].
$$
Recall that $F$ is not defined on $\sca_{flat}$, so when we write $F(\cca)$ for some subset $\cca\subset\sca$ we always mean $F(\cca\smallsetminus\sca_{flat})$.

An element $\mu\in\sca$ is called a {\it solvsoliton} when $\Ricci_\mu=cI+D$ for some $c\in\RR$ and $D\in\Der(\mu)$.  Any solvsoliton belongs to $\sca_\RR$ (see \cite[Theorem 4.8]{solvsolitons}) and any Ricci soliton metric on a solvable Lie group is isometric to a solvsoliton (see \cite{Jbl}).

The only non-abelian Lie groups with $m_\mu=M_\mu$ are precisely those admitting a unique metric up to isometry and scaling, i.e.\
$$
\mu_{heis}(e_1,e_2)=e_3, \qquad \mu_{hyp}(e_n,e_i)=e_i, \quad i=1,\dots,n-1,
$$
and zero otherwise.  Note that $S_{\mu_{hyp}}$ is isometric to the real hyperbolic space $\RR H^n$ and so $m_{\mu_{hyp}}=M_{\mu_{hyp}}=n$.  On the other hand, $m_{\mu_{heis}}=M_{\mu_{heis}}=\frac{1}{3}$, and since $\mu_{heis}\in \overline{\Gl(\sg)\cdot\mu}$ for any $\mu\notin\Gl(\sg)\cdot\mu_{hyp}$, we have that
\begin{quote}
$m_\mu\leq\frac{1}{3}$ for any $\mu\in\sca$ such that $\mu\notin\Gl(\sg)\cdot\mu_{hyp}$.
\end{quote}

The maximum of $F$ among all left-invariant metrics on a nilpotent Lie group is attained at a nilsoliton, which is known to be unique up to isometry and scaling, if one exists.  For any nonzero $\mu\in\nca$,
$$
F(\Gl(\sg)\cdot\mu)=\left\{
\begin{array}{ll}
\left(\frac{1}{3},M_\mu\right], & S_\mu\;\mbox{admits a nilsoliton and}\, \mu\notin\Gl(\sg)\cdot\mu_{heis}, \\ \\
\left(\frac{1}{3},M_\mu\right), & S_\mu\;\mbox{does not admit any nilsoliton}, \\ \\
\left\{\frac{1}{3}\right\}, & \mu\in\Gl(\sg)\cdot\mu_{heis}.
\end{array} \right.
$$
We also have that $F(\nca)=[\frac{1}{3},C_n]$ for some constant $C_n<n$ depending only on $n$, which is necessarily the value of $F$ at some nilsoliton.  The nilsolitons with $\Ricci=\Diag(1,2,\dots,n)$ have $F=\frac{n(n-1)}{2(2n+1)}$, showing that $\frac{1}{5}n\leq C_n$ for all $7\leq n$.  In the nilpotent case, the functional $F$ is strictly increasing along any Ricci flow solution $g(t)$, unless $g(0)$ is a nilsoliton (see \cite{nilricciflow}).

As in Section \ref{alm-abel-preli}, in order to study the almost-abelian case, one fixes an orthogonal decomposition $\sg=\ngo\oplus\RR e_n$ and attaches to each matrix $A\in\glg_{n-1}(\RR)$ (identified with $\glg(\ngo)$ via any fixed orthonormal basis) the Lie bracket $\mu_A$ defined by $\mu_A(\ngo,\ngo)=0$ and $\ad_{\mu_A}{e_n}|_{\ngo}=A$.  The construction covers, up to isometry, all left-invariant metrics on almost abelian Lie groups. Note that the class of almost-abelian Lie brackets is contained in $\sca_{\im\RR}\cup\sca_\RR$.

Using well-known formulas for the Ricci curvature of solvmanifolds, one obtains that
$$
F(A) \leq 1+\frac{(\tr{A})^2}{\tr{S(A)^2}} \leq n,
$$
where $S(A):=\unm(A+A^t)$.  Moreover, $F(A)=n$ if and only if $S(A)=aI$, $a\ne 0$, if and only if $\mu_A$ is isometric to the real hyperbolic space $\RR H^n$.  It was proved in
\cite[Proposition 3.3]{Arr} that $A$ is a solvsoliton if and only if either $A$ is normal or $A$ is nilpotent and $[A,[A,A^t]]=cA$ for some $c\in\RR$.

In the case when $\tr{A}=0$, i.e.\ $\mu_A$ unimodular, it follows that
$$
F(A)=\frac{\left(\tr{S(A)^2}\right)^2}{\left(\tr{S(A)^2}\right)^2 + \unc|[A,A^t]|^2},
$$
hence $F(A)\leq 1$ and equality holds if and only if $[A,A^t]=0$.  Thus $M_A=1$ for any $\mu_A\in\sca_\RR$ and it is a maximum if and only if $A$ is semisimple.  Note that the maxima of $F$ on one of these Lie groups are precisely solvsolitons, as in the nilpotent case.

More generally, it is proved in \cite[(2)]{BhmLfn2} that for any $\mu\in\sca_{unim}$, $F(\lambda)=M_\mu$ for some $\lambda\in\Gl(\sg)\cdot\mu$ if and only if $\lambda$ is a solvsoliton.

Some general results on Ricci pinching of solvmanifolds follow.

\begin{theorem}\label{main}\cite{RP}
\quad
\begin{itemize}
  \item[(i)] $0<m_\mu$ and $F\left(\overline{\Gl(\sg)\cdot\mu}\right)=[m_\mu,M_\mu]$ for any $\mu\in\sca_\RR$.
  \item[ ]
  \item[(ii)] For $n\geq 4$, $\inf\{ m_\mu:\mu\in\sca_\RR\}=0$; in particular, $F(\sca_\RR)=(0,n]$.
  \item[ ]
  \item[(iii)] $F(\Gl(\sg)\cdot\mu)=(m_\mu,M_\mu]$ for every $\mu\in\sca_{\RR}\cap\sca_{unim}$.
  \item[ ]
  \item[(iv)] $m_\mu=0$ for any $\mu\in\sca\smallsetminus\sca_\RR$.
\end{itemize}
\end{theorem}

\section{Ricci pinching of $G_2$-structures on solvmanifolds}\label{RP-g2-sec}

With Sections \ref{erp} and \ref{RP-sec} as our motivation, we now study the extremal points and values of the functional
$$
F:\sca_{closed}\smallsetminus\sca_{tf} \longrightarrow \RR, \qquad F(\mu):=\frac{\scalar_\mu^2}{|\Ricci_\mu|^2},
$$
where $\scalar_\mu$ and $\Ricci_\mu$ are respectively the scalar curvature and Ricci operator of $(S_\mu,\ip)$.  Since no solvable Lie group admits an Einstein (non-flat) and closed $G_2$-structure (see \cite{FrnFinMnr2}),
$$
0<F(\mu)<7, \qquad\forall\mu\in\sca_{closed}\smallsetminus\sca_{tf}.
$$
For each $\mu\in\sca_{closed}$ we define,
$$
n_\mu:=\inf F(\Gl(\sg)\cdot\mu\cap\sca_{closed}), \qquad N_\mu:=\sup F(\Gl(\sg)\cdot\mu\cap\sca_{closed}),
$$
that is, the infimum and supremum of $F$ among all closed $G_2$-structures on the Lie group $S_\mu$.  Recall that $F$ is not defined on $\sca_{tf}$, so when we write $F(\cca)$ for some subset $\cca\subset\sca_{closed}$ we always mean $F(\cca\smallsetminus\sca_{tf})$.

It follows from Theorem \ref{bryant} that if $\mu\in\sca_{closed}$ and $S_\mu$ admits a lattice (i.e.\ a cocompact discrete subgroup), then $F(\mu)\leq 3$, and equality holds if and only if $\mu$ is ERP.  On the other hand, if for a unimodular $\mu\in\sca_{closed}$ there is a solvsoliton $\lambda\in\Gl_7(\RR)\cdot\mu\cap\sca_{closed}$, then $F(\lambda)=N_\mu=M_\mu$.

For any class of $G_2$-structures defined by a closed cone $\cca\subset\sca$ such that $\cca\cap\sca_{tf}=\{ 0\}$, one has that
$$
\cca\cap\sca_{closed}\cap\{\mu:|\mu|=1\}
$$
is a compact subset of $\sca_{closed}\smallsetminus\sca_{tf}$.  This implies that the infimum and supremum of $F(\cca\cap\sca_{closed})$ are actually minimum and maximum, respectively, and
$$
0 < \min F(\cca\cap\sca_{closed}) \leq \max F(\cca\cap\sca_{closed}) < 7.
$$
Examples of classes $\cca$ for which the above holds include
\begin{itemize}
\item $\nca$ or any closed cone contained in $\nca$.
\item[ ]
\item $\overline{\Gl_7(\RR)\cdot\mu}$ for any $\mu\in\sca_\RR$ (see Lemma \ref{ramiro}).
\end{itemize}

It would be really interesting to know the number
$$
N_{closed}:=\sup F(\sca_{closed}).
$$
If $N_{closed}$ turns out to be a maximum, then the closed $G_2$-structures with $F=N_{closed}$ should be special in some sense.  At the moment, the largest known value for $F$ on $\sca_{closed}$ is $\frac{81}{17}\sim 4.76$ and was found in \cite[Example 4.11]{LS-ERP} at a shrinking Laplacian soliton.

\begin{proposition}\label{main-g2}
\quad
\begin{itemize}
  \item[(i)] $0<n_\mu$ for any $\mu\in\sca_{closed}\cap\sca_\RR$.
  \item[ ]
  \item[(ii)] $\inf\{ n_\mu:\mu\in\sca_{closed}\cap\sca_\RR\}=0$.
\end{itemize}
\end{proposition}

\begin{proof}
Given $\mu\in\sca_{closed}\cap\sca_{\RR}$, it follows from \cite[Lemma 3.4]{BhmLfn} that
$$
\overline{\Gl_7(\RR)\cdot\mu}\cap\sca_{closed}\cap\{\mu:|\mu|=1\}
$$
is a  compact subset of $\sca_{closed}\smallsetminus\sca_{tf}$, so part (i) follows.  Part (ii) was proved in Example \ref{alm-abel-g2} by using the family $C_t$.
\end{proof}

We note that part (i) also follows from Theorem \ref{main}, (i) and the fact that $m_\mu\leq n_\mu$.

\begin{corollary}
For any non-abelian solvable Lie group $S$ of real type there exists a constant $C(S)>0$ depending only on $S$ such that
$$
|\Ricci(\psi)| \leq C(S)|\scalar(\psi)|,
$$
for any left-invariant $G_2$-structure $\psi$ on $S$.
\end{corollary}

This estimate may have some applications in the study of convergence of geometric flows for $G_2$-structures (see \cite{BhmLfn}).

\subsection{Nilpotent case}\label{nilp-g2}
Since $\mu_{heis}$ does not appear in the list $\mu_1,\dots,\mu_{12}$ given in Section \ref{nilp-preli}, we obtain that $\frac{1}{3}<\min F(\sca_{closed}\cap\nca)$ and so $\frac{1}{3}<n_\mu$ for any $\mu\in\sca_{closed}\cap\nca$.  In what follows, we describe what we know about the behavior of $F$ on each of the nilpotent Lie groups admitting a closed $G_2$-structure (see \cite{Ncl}):

\begin{itemize}
  \item $\mu_1$: $F$ is not defined.

  \item $\mu_2$: $F\equiv\unm$; in particular, $n_{\mu_2}=N_{\mu_2}=\unm$.

  \item $\mu_3$:  $F\equiv \unm$ on the curve of closed Laplacian solitons.

  \item $\mu_4$: $F=\frac{4}{5}= N_{\mu_4}$ at the nilsoliton and $F=\frac{3}{4}$ at the Laplacian soliton.

  \item $\mu_5$:  $F=\frac{3}{4}$ at the Laplacian soliton, but $F>\frac{3}{4}$ on a certain curve of closed $G_2$-structures.

  \item $\mu_6$: At the nilsoliton, $F=\frac{4}{5}=N_{\mu_6}$, and at the Laplacian soliton, $F=\frac{3}{4}$.

  \item $\mu_7$: At the Laplacian soliton, $F=\frac{3}{4}$, though $F>\frac{3}{4}$ on a curve of closed $G_2$-structures.

  \item $\mu_{12}$: $F=1= N_{\mu_{12}}$ at the nilsoliton.
\end{itemize}

We note that Laplacian solitons in general fail to provide the maximum value of $F$ on a given nilpotent Lie group.

\subsection{Almost-abelian case}\label{alm-abel-g2}
We work in this section on the class of almost-abelian solvable Lie groups (see Section \ref{alm-abel-preli}).  For each $A\in\slg_3(\CC)$ one has that
$$
F(A)=\frac{|H(A)|^4}{|H(A)|^4 + \frac{1}{8}|[A,A^*]|^2},
$$
where $H(A):=\unm(A+A^*)$ is the hermitian part of $A$ and $|B|^2:=\tr{BB^*}$ for any $B\in\slg_3(\CC)$.  It follows that $F(A)\leq 1$ and equality holds if and only if $[A,A^*]=0$.  Thus the maximum of $F$ on a given non-nilpotent $S_{\mu_A}$ is only attained if $A$ is semisimple and it is both a Laplacian and a Ricci soliton.  The following example explicitly shows that the maximum value of $F$ is not always attained at a Laplacian soliton in the nilpotent case.

\begin{example}\label{mu6-F}
Consider the set of closed $G_2$-structures on the $3$-step nilpotent Lie group $S_{\mu_6}$ parameterized by
$$
\left[\begin{matrix} 0&a&0\\ &0&1 \\ &&0 \end{matrix}\right], \qquad a>0.
$$
The Ricci soliton and the Laplacian soliton correspond to $a=1$ and $a=\sqrt{2}$, respectively.  We have that
$$
F(a)=\frac{a^4+2a^2+1}{2a^4+a^2+2}, \qquad 0<a,
$$
a function with only one critical point, a global maximum with $F=\frac{4}{5}$ at the nilsoliton $a=1$.  Note that at the Laplacian soliton, $F(\sqrt{2})=\frac{3}{4}$.
\end{example}

Concerning the behavior of $F$ close to $\sca_{tf}$, we have that
$$
A_t:=\left[\begin{matrix} t&-1&\\ 1&-t& \\ &&0\end{matrix}\right], \quad F(A_t)=\frac{t^4}{t^4+t^2} \underset{t\to 0}\longrightarrow 0; \qquad
B_t:=\left[\begin{matrix} t&-1&\\ 1&t&\\ &&-2t\end{matrix}\right], \quad F(B_t)\equiv 1.
$$
This implies that $F$ diverges at the torsion-free $G_2$-structure $A_0=B_0$.  Since $\Spec(A_t)=\{\pm\im\sqrt{1-t^2}\}$ for any $t<1$, we deduce that $n_{A_0}=0$.   On the other hand, $\Spec(B_t)=\{\pm\im+t,-2t\}$, so the family of Laplacian solitons $\mu_{B_t}$ is pairwise non-isomorphic.

More generally, $n_A=0$ for every $\mu_A\in\sca_{\im\RR}\smallsetminus\nca$.  Indeed, if $a\ne b$, then
$$
D_t:=\left[\begin{matrix} a\im&t&\\ &b\im& \\ &&c\im \end{matrix}\right], \qquad F(D_t)=\frac{t^4}{t^4 + (a-b)^2t^2} \underset{t\to 0} \longrightarrow 0.
$$
Recall that the class of almost-abelian Lie brackets is contained in the disjoint union of $\sca_\RR\smallsetminus\nca$, $\nca$ and $\sca_{\im\RR}\smallsetminus\nca$.  In the list of matrices given in Section \ref{alm-abel-preli}, the first two belong to $\sca_\RR\smallsetminus\nca$, the second two to $\nca$ and the last two to $\sca_{\im\RR}\smallsetminus\nca$.

The following family $C_t$, $0<t$, in $\sca_\RR$ given by
$$
C_t:=\left[\begin{matrix} t&-1&\\ 1&0& \\ &&-t \end{matrix}\right], \qquad F(C_t)=\frac{4t^4}{4t^4 + t^2} \underset{t\to 0} \longrightarrow 0,
$$
shows that $\inf\{ n_A:\mu_A\in\sca_{closed}\cap\sca_\RR\}=0$.

In \cite[Example 5.20]{LF}, the Laplacian flow on the family
$$
\left[\begin{matrix} 0&a&\\ b&0& \\ &&0 \end{matrix}\right], \qquad F(a,b)=\frac{(a+b)^4}{(a+b)^4 + (a^2-b^2)^2},
$$
was studied.  Using the ODE obtained there for $a(t),b(t)$, it is easy to prove that $F$ is strictly decreasing along the Laplacian flow solutions starting at closed $G_2$-structures with $ab<0$, $a\ne-b$.  This shows that the Laplacian flow does not always improve the Ricci pinching of closed $G_2$-structures.  On the other hand, the functional $F$ was found to be increasing in some other Laplacian flow solutions like in the above example with $ab>0$ and in the evolution studied in \cite[Example 4.9]{LS-ERP}.

\subsection{Open questions}
It would be interesting to know the answers to the following natural questions:

\begin{itemize}
\item Given $A_0\in\slg_3(\CC)$ such that $\Spec(A_0)\subset\im\RR$, i.e.\ $\mu_{A_0}\in\sca_{\im\RR}$, does $\lim F$ as $A$ goes to $A_0$ exist on the isomorphism class $\RR^*\Sl_3(\CC)\cdot A_0$?  Examples $A_t$ and $B_t$ above show that such a limit does not exist on the set of all closed almost-abelian Lie brackets.
\item[ ]
\item Is $n_\mu=0$ for any $\mu\notin\sca_{closed}\cap\sca_\RR$?  This holds in the Riemannian case (see Theorem \ref{main}, (iv)) and in the $G_2$ case for almost-abelian Lie groups (see Section \ref{alm-abel-g2}).
\item[ ]
\item Is $F(\Gl_7(\RR)\cdot\mu\cap\sca_{closed})=(n_\mu,N_\mu]$ for any $\mu\in\sca_\RR$?
\item[ ]
\item Does $F(\lambda)=N_\mu$ hold for any solvsoliton $\lambda\in\Gl_7(\RR)\cdot\mu\cap\sca_{closed}$?  This is known to be true in the unimodular case and it is open in the non-unimodular Riemannian case.
\item[ ]
\item Is $F(\Gl(\sg)\cdot\mu)=(0,M_\mu)$  for any $\mu\in\sca_{closed}\smallsetminus\sca_\RR$?  What about for $\mu\in\sca_{\im\RR}\cap\sca_{closed}$?
\item[ ]
\item What is the value of $\sup\{ M_\mu:\mu\in\sca_{\im\RR}\cap\sca_{closed}\}$?
\end{itemize}

{\small
\begin{figure}
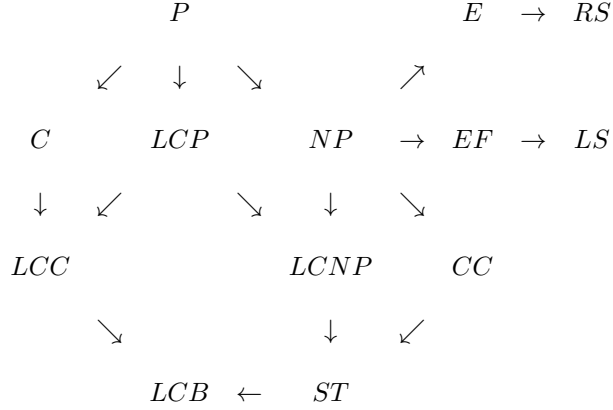

$$
\begin{array}{ccccccccc}
&&P&&&&E&\rightarrow&RS \\ \\
&\swarrow & \downarrow & \searrow &&\nearrow& && \\ \\
C&&LCP&&NP&\rightarrow&EF&\rightarrow& LS \\ \\
\downarrow & \swarrow&&\searrow & \downarrow &\searrow&&& \\ \\
LCC &&&& LCNP && CC && \\ \\
&\searrow&&& \downarrow & \swarrow &&&\\ \\
&&LCB&\leftarrow& ST &&&&
\end{array}
$$
\caption{Special classes of $G_2$-structures}\label{g2-fig}
\end{figure}}

\section{Appendix: Special classes of $G_2$-structures}\label{dist-G2}

The {\it torsion forms} of a $G_2$-structure $\vp$ on $M$ are the components of the {\it intrinsic torsion} $\nabla\vp$, where $\nabla$ is the Levi-Civita connection of the metric $g$ attached to $\vp$.  They can be defined as the unique differential forms $\tau_i\in\Omega^iM$, $i=0,1,2,3$, such that
\begin{equation}\label{dphi}
d\vp=\tau_0\ast\vp+3\tau_1\wedge\vp+\ast\tau_3, \qquad d\ast\vp=4\tau_1\wedge\ast\vp+\tau_2\wedge\vp.
\end{equation}

Some special classes of $G_2$-structures are defined or characterized as follows, we refer to \cite{FrnGry} for further information:

\begin{itemize}
\item {\it parallel} (P) or {\it torsion-free}: $d\vp=0$ and $d\ast\vp=0$, or equivalently, $\nabla\vp=0$ (for $M$ compact, this is equivalent to $\vp$ {\it harmonic} (H), i.e.\  $\Delta\vp=0$);
\item[ ]
\item {\it closed} (C) or {\it calibrated}: $d\vp=0$;
\item[ ]
\item {\it coclosed} (CC) or {\it cocalibrated}: $d\ast\vp=0$;
\item[ ]
\item {\it locally conformal parallel} (LCP): $d\vp=3\tau_1\wedge\vp$ and $d\ast\vp=4\tau_1\wedge\ast\vp$;
\item[ ]
\item {\it locally conformal closed} (LCC): $d\vp=3\tau_1\wedge\vp$ (in particular, $d\tau_1=0$);
\item[ ]
\item {\it nearly parallel} (NP): $d\vp=\tau_0\ast\vp$ (which implies that $\Delta\vp=\tau_0^2\vp$ and $\Ricci=\frac{3}{8}\tau_0^2g$);
\item[ ]
\item {\it locally conformal nearly parallel} (LCNP): $d\vp=\tau_0\ast\vp+3\tau_1\wedge\vp$ and $d\ast\vp=4\tau_1\wedge\ast\vp$;
\item[ ]
\item {\it skew-torsion} (ST) or $G_2T$-{\it structures}: $d\ast\vp=4\tau_1\wedge\ast\vp$ (in particular, $d\tau_1=0$);
\item[ ]
\item {\it locally conformal balanced} (LCB): $d\tau_1=0$;
\item[ ]
\item {\it eigenform} (EF): $\Delta\vp=c\vp$ for some $c\in\RR$;
\item[ ]
\item {\it Einstein} (E): $\ricci=cg$ for some $c\in\RR$;
\item[ ]
\item {\it Laplacian soliton} (LS): $\Delta\vp = c\vp + \lca_X\vp$ for some $c\in\RR$ and $X\in\mathfrak{X}(M)$ (called {\it expanding}, {\it steady} or {\it shrinking} if $c>0$, $c=0$ or $c<0$, respectively);
\item[ ]
\item {\it Ricci soliton} (RS): $\ricci = cg + \lca_Xg$ for some $c\in\RR$ and $X\in\mathfrak{X}(M)$.
\end{itemize}

We also refer to \cite{Rff} for a more detailed study of most of these classes of $G_2$-structures and their possible intersections.  Figure \ref{g2-fig} describes the obvious inclusions among them.  For a given class $\cca$, a $G_2$-structure $\vp$ is said to be {\it locally conformal} $\cca$ if for each $p\in M$, there exist an open neighborhood $U$ and a conformal change $\psi:=e^f\vp$, $f\in C^{\infty}(U)$ such that $(U,\psi)$ is a $G_2$-structure of class $\cca$.

\end{document}